\documentclass{amsart}

\usepackage{amsmath,enumerate,amsfonts}
\usepackage{amsthm}

\newtheorem{lemma}{Lemma}[section]

\newtheorem{theorem}[lemma]{Theorem}
\newtheorem{corollary}[lemma]{Corollary}
\theoremstyle{definition}

\newtheorem{conjecture}[lemma]{Conjecture}

\theoremstyle{remark}

\title[On a dimension formula for self-affine measures]{On Falconer's formula for the generalised R\'enyi dimension of a self-affine measure}
\author{Ian D. Morris}

\begin{document}

\maketitle

\begin{abstract}
We investigate a formula of K. Falconer which describes the typical value of the generalised R\'enyi dimension, or generalised $q$-dimension, of a self-affine measure  in terms of the linear components of the affinities. We show that in contrast to a related formula for the Hausdorff dimension of a typical self-affine set, the value of the generalised $q$-dimension predicted by Falconer's formula varies discontinuously as the linear parts of the affinities are changed. Conditionally on a conjecture of J. Bochi and B. Fayad, we show that the value predicted by this formula for pairs of two-dimensional affine transformations is discontinuous on a set of positive Lebesgue measure. These discontinuities derive from discontinuities of the lower spectral radius which were previously observed by the author and J. Bochi.

\end{abstract}


\section{Introduction}

If $\mathsf{T}:=(T_1,\ldots,T_N)$ is a finite collection of transformations of a complete metric space $X$, and each $T_i$ is a contraction in the sense that for some $\lambda<1$  one has $d(T_ix,T_iy) \leq \lambda d(x,y)$ for all $x,y \in X$, it is well-known that there exists a unique nonempty compact set $Z_{\mathsf{T}}\subseteq X$ such that
\[Z_{\mathsf{T}}=\bigcup_{i=1}^N T_i\left(Z_{\mathsf{T}}\right),\]
see for example \cite{Fa03,Hu81}. The case in which $X=\mathbb{R}^d$ and the transformations $T_i$ are affine  -- in which case $Z_{\mathsf{T}}$ is called a \emph{self-affine set} -- has been the subject of intense study over the last few decades (for a recent survey we note \cite{Fa13}). 

Various notions of fractal dimension have been investigated for both general and special classes of self-affine set. In certain special cases where the linear parts of the affinities preserve or permute the horizontal and vertical axes of $\mathbb{R}^2$, explicit formul{\ae}
for the Hausdorff dimension and box dimension exist (see e.g. \cite{Ba07,Be84,Fr12,LaGa92,Mc84}). In the context of general self-affine sets, a landmark result of K. Falconer \cite{Fa88} established the Hausdorff and box dimensions of ``typical'' self-affine sets in a sense which we now describe. Let $M_d(\mathbb{R})$ denote the set of all $d \times d$ real matrices, and recall that for $A \in M_d(\mathbb{R})$ we define the \emph{singular values} of $A$, denoted $\sigma_1(A),\ldots,\sigma_d(A)$, to be the non-negative square roots of the eigenvalues of the positive semi-definite matrix $A^*A$, listed in decreasing order. Let us define a function $\varphi \colon (0,+\infty) \times M_d(\mathbb{R}) \to [0,+\infty)$ by
\[\varphi^s(A):=\left\{\begin{array}{cl}\sigma_1(A)\cdots \sigma_{k}(A)\sigma_{k+1}(A)^{s-k},&k \leq s \leq k+1 \leq d\\
|\det A|^\frac{s}{d},&s \geq d,\end{array}\right.\]
and for every $A_1,\ldots,A_N$ in the open unit ball of $M_d(\mathbb{R})$ define the \emph{affinity dimension} or \emph{singularity dimension} of $(A_1,\ldots,A_N)$ to be the quantity
\[\mathfrak{s}(A_1,\ldots,A_N):=\inf\left\{s>0\colon  \sum_{n=1}^\infty \sum_{i_1,\ldots,i_n=1}^N \varphi^s\left(A_{i_1}\cdots A_{i_n}\right)<\infty\right\}.\]
Falconer showed that for any fixed invertible $d \times d$ matrices $A_1,\ldots,A_N$ with Euclidean norm strictly less than $\frac{1}{3}$, for Lebesgue almost all $v_1,\ldots,v_N \in \mathbb{R}^d$ the self-affine set $Z_{\mathsf{T}}$ associated to the collection of affine transformations $\mathsf{T}:=(T_1,\ldots,T_N)$ defined by $T_ix:=A_ix+v_i$ satisfies
\[\mathrm{dim}_H\left(Z_{\mathsf{T}}\right) =\mathrm{dim}_B\left(Z_{\mathsf{T}}\right)  =\min\left\{\mathfrak{s}(A_1,\ldots,A_N),d\right\},\]
see \cite{Fa88}; the bound on the norm was subsequently relaxed to $\frac{1}{2}$ by B. Solomyak \cite{So98}, and to $1$ by T. Jordan, K. Simon and M. Pollicott for a notion of ``almost self-affine set'' which incorporates additional random translations \cite{JoPoSi07}. While it is well-known that the Hausdorff dimension of $Z_{\mathsf{T}}$ can fail to depend continuously on the affinites $T_1,\ldots,T_N$ (see e.g. \cite[Example 9.10]{Fa03}), it was recently shown by D.-J. Feng and P. Shmerkin in \cite{FeSh14} that the affinity dimension $\mathfrak{s}$ is a continuous function of $(A_1,\ldots,A_N)$. An alternative proof of this statement was subsequently given by the author \cite{Mo16}. 

In this article we will focus not on self-affine sets but on self-affine \emph{measures}. A Borel probability measure $\mu$ on $\mathbb{R}^d$ is called self-affine if there exist a probability vector $\mathsf{p}=(p_1,\ldots,p_N)$ and a collection of affinities $\mathsf{T}=(T_1,\ldots,T_N)$ such that
\[\mu(A)=\sum_{i=1}^Np_i \mu\left(T_i^{-1}A\right)\]
for every Borel set $A \subseteq \mathbb{R}^d$. If $T_1,\ldots,T_N$ are contractions then for each probability vector $\mathsf{p}$ with all probabilities nonzero there exists a unique Borel probability measure $\mu_{\mathsf{p},\mathsf{T}}$ satisfying the above functional equation (see e.g. \cite[Theorem 2.8]{Fa97}), and the support of that measure is equal to the associated self-affine set $Z_{\mathsf{T}}$. In this article our interest is in the \emph{generalised $q$-dimension} or \emph{generalised R\'enyi dimension} of a self-affine measure, which is defined as follows. For each $r>0$ let $\mathcal{M}_r$ denote the set of all $r$-mesh cubes on $\mathbb{R}^d$, that is, the set of all $d$-dimensional cubes of the form $[j_1r,(j_1+1)r) \times [j_2r, (j_2+1)r) \times \cdots \times [j_dr,(j_d+1)r)$ where $j_1,\ldots,j_d \in \mathbb{Z}$. For $q > 1$ we define
\[M_r(q,\mu):= \sum_{C \in \mathcal{M}_r} \mu(C)^q\]
for every $r>0$, and
\[\underline{D}_q(\mu) := \liminf_{r \to 0} \frac{\log M_r(q,\mu)}{(q-1)\log r},\qquad\overline{D}_q(\mu) := \limsup_{r \to 0} \frac{\log M_r(q,\mu)}{(q-1)\log r}.\]
If $\underline{D}_q(\mu)$ and $\overline{D}_q(\mu)$ are equal then we define the generalised $q$-dimension of $\mu$ to be their common value and denote this by $D_q(\mu)$. For $q>1$ the generalised $q$-dimension admits an alternative expression as a limit of certain integrals \cite{La95}. In \cite[Theorem 6.2]{Fa99}, K. Falconer characterised the generalised R\'enyi dimensions of typical self-affine measures in a similar manner to his earlier characterisation of the Hausdorff and box dimensions of typical self-affine sets:
\begin{theorem}[Falconer]
Let $(A_1,\ldots,A_N)$ be invertible linear transformations of $\mathbb{R}^d$ such that $\|A_i\|<\frac{1}{2}$ for every $i$, let $\mathsf{p}=(p_1,\ldots,p_N)$ be a probability vector with all entries nonzero, and for each $q>1$ define $\mathfrak{r}_q(A_1,\ldots,A_N,\mathsf{p})$ to be the quantity
\[\sup\left\{s>0 \colon \sum_{n=1}^\infty \sum_{i_1,\ldots,i_n=1}^N\varphi^s(A_{i_1}\cdots A_{i_n})^{1-q} p_{i_1}^q \cdots p_{i_n}^q<\infty\right\}.\]
If $1<q\leq 2$ then for Lebesgue-almost-every $(v_1,\ldots,v_N)\in\mathbb{R}^{Nd}$ the self-affine measure $\mu_{\mathsf{p},\mathsf{T}}$ corresponding to the transformations $T_ix:=A_ix+v_i$ and the probability vector $\mathsf{p}$ satisfies $D_q(\mu_{\mathsf{p},\mathsf{T}})=\min\{\mathfrak{r}_q(A_1,\ldots,A_N,\mathsf{p}),d\}$. 
\end{theorem}
In the later article \cite{Fa10} this result was extended to a more general class of almost self-affine measures, under the weaker hypotheses $\max\|A_i\|<1$ and $q>1$ but requiring randomised translations in a similar manner to \cite{JoPoSi07}. In view of the recent work of Feng and Shmerkin \cite{FeSh14} on the continuity of the formula $\mathfrak{s}(A_1,\ldots,A_N)$ for the typical dimension of a self-affine set, it is natural to ask whether the formula $\mathfrak{r}_q(A_1,\ldots,A_N,\mathsf{p})$ for the typical dimension of a self-affine measure is also continuous with respect to changes in the matrices $A_1,\ldots,A_N$. The purpose of this article is to answer this question negatively.

In order to state our results we require an additional definition. Let us say that a pair of matrices $(A_1,A_2)$ is $(c,\varepsilon,\lambda)$-resistant if it has the following property: for all choices of $i_1,\ldots,i_n\in\{1,2\}$ such that at most $\varepsilon n$ of the integers $i_k$ are equal to $2$, we have $\|A_{i_1}\cdots A_{i_n}\|\geq c\lambda^n$. We will say that $(A_1,A_2)$ \emph{resists impurities}, or more simply is \emph{resistant}, if it is $(c,\varepsilon,\lambda)$-resistant for some $c,\varepsilon>0$ and some $\lambda>1$. We recall the following conjecture of Bochi and Fayad \cite{BoFa06}:
\begin{conjecture}[Bochi-Fayad Conjecture]\label{co:bf}
Let $\mathcal{H}$ denote the set of all $2 \times 2$ real matrices with unit determinant and unequal real eigenvalues, and let $\mathcal{E}$ denote the set of all $2 \times 2$ real matrices with unit determinant and non-real eigenvalues. Then the set of resistant pairs $(A_1,A_2) \in \mathcal{H}\times \mathcal{E}$ has full Lebesgue measure.
\end{conjecture}
Some partial results in the direction of Conjecture \ref{co:bf} may be found in \cite{AlSeUn11,AvRo09,FaKr08}. C. Bonatti has constructed explicit examples of resistant pairs in which $A_2$ is a rational rotation, and these examples are described in \cite{BoMo15}.

When investigating the discontinuities of $\mathfrak{r}_q$ we will focus on the situation in which $(A_1,\ldots,A_N)$ is a pair of real matrices of dimension two. In this case the probability vector $\mathsf{p}=(p_1,p_2)$ has the form $(p,1-p)$ for some real number $p\in (0,1)$, and in view of this we shall simply write $\mathfrak{r}_q(A_1,A_2,p)$ in place of the value $\mathfrak{r}_q(A_1,A_2,(p,1-p))$ defined previously. We prove:
\begin{theorem}\label{th:fad}
The function $\mathfrak{r}$ admits the following discontinuities:
\begin{enumerate}[(i)]
\item
Let $q>1$, $p \in (0,1)$ and $0<\delta<\lambda<\frac{1}{2}$, and suppose that $\delta$ is small enough that
\[\frac{\log\min\{p^q,(1-p)^q\}}{\log \sqrt{\lambda\delta}} < \frac{\log (p^q+(1-p)^q)}{\log \lambda}.\]
Then the function $(A_1,A_2)\mapsto \mathfrak{r}_q(A_1,A_2,p)$ is discontinuous at the pair
\[A_1:=\left(\begin{array}{cc}\lambda &0\\0&\delta \end{array}\right),\qquad A_2:=\left(\begin{array}{cc}\lambda &0\\0&\lambda\end{array}\right).\]
\item
If  the Bochi-Fayad Conjecture is true then there exists a set $X \subseteq M_2(\mathbb{R})^2$ with positive Lebesgue measure  with the following properties: $\|A_1\|,\|A_2\|<\frac{1}{2}$ for all $(A_1,A_2) \in X$, and there exists $Q>1$ such that for all $p \in [\frac{1}{2},1)$ and $q>Q$, the function $(A_1,A_2)\mapsto \mathfrak{r}_q(A_1,A_2,p)$ is discontinuous at every point of $X$.
\end{enumerate}
\end{theorem}
\emph{Remark.}
The reader will see from the proofs below that Theorem \ref{th:fad}(i) in fact has the following more precise statement: if $R_\theta$ denotes the matrix corresponding to rotation about the origin through angle $\theta$, then
\[\liminf_{\theta \to 0} \mathfrak{r}_q(A_1,\lambda R_\theta,p)<\mathfrak{r}_q(A_1,A_2,p).\]
Clearly the case of $p<\frac{1}{2}$ may also be considered in (ii) by interchanging the r\^oles of the matrices $A_1$ and $A_2$.

It would be of interest to remove the restriction on $q$ in (ii) so as to bring that statement into line with (i). Unfortunately the Bochi-Fayad Conjecture does not seem to be a sufficiently strong statement to allow us to deduce that $(A_1,A_2) \mapsto \mathfrak{r}_q(A_1,A_2,p)$ has a positive-measure set of discontinuities for every $p \in (0,1)$ and $q>1$. We nonetheless conjecture that this map has a positive-measure set of discontinuities for all such $p$ and $q$, and hope that whatever methods may be employed to prove the Bochi-Fayad Conjecture will also suffice to establish the discontinuity of $\mathfrak{r}_q(\cdot,\cdot,p)$ on a set of positive measure for all $q>1$ and $p \in (0,1)$.


\section{Proof of Theorem \ref{th:fad}}
For $p \in (0,1)$, $q>1$, $s>0$ and invertible matrices $A_1,A_2 \in M_2(\mathbb{R})$, let us define $p_1:=p$ and $p_2:=(1-p)$, and write
\[\mathbf{R}_q(A_1,A_2,p,s):=\lim_{n \to \infty} \frac{1}{n}\log\left(\sum_{i_1,\ldots,i_n=1}^N \varphi^s\left(A_{i_1}\cdots A_{i_n}\right)^{1-q}p_{i_1}^q \cdots p_{i_n}^q\right).\]
We note:
\begin{lemma}
For fixed invertible matrices $A_1,A_2$ such that $\|A_1\|,\|A_2\|<1$, fixed $q>1$ and fixed $p \in(0,1)$ the function $\mathbf{R}_q(A_1,A_2,p,\cdot) \colon (0,+\infty) \to \mathbb{R}$ is well-defined and strictly increasing.
\end{lemma}
\begin{proof}
It is well-known that $\varphi^s(AB) \leq \varphi^s(A)\varphi^s(B)$ for all $s>0$ and $A,B \in M_d(\mathbb{R})$, see for example \cite[Lemma 2.1]{Fa88}; since the proof is brief we include it. For $s \geq d$ the result is trivial, and for $k \leq s<k+1$, $k=0,\ldots,d-1$ we have
\begin{align*}\varphi^s(AB) &= \left(\sigma_1(AB)\cdots \sigma_{k+1}(AB)\right)^{s-k}  \left(\sigma_1(AB)\cdots \sigma_{k}(AB)\right)^{k+1-s}\\
&=\left\|\wedge^{k+1}(AB)\right\|^{s-k}\left\|\wedge^k(AB)\right\|^{k+1-s}\\
&\leq\left\|\wedge^{k+1}A\right\|^{s-k}\left\|\wedge^{k+1}B\right\|^{s-k}\left\|\wedge^kA\right\|^{k+1-s}\left\|\wedge^kB\right\|^{k+1-s} =\varphi^s(A)\varphi^s(B)\end{align*}
as claimed. It follows that each sequence $(a_n)$ defined by
\[a_n:=\log\left(\sum_{i_1,\ldots,i_n=1}^N \varphi^s\left(A_{i_1}\cdots A_{i_n}\right)^{1-q}p_{i_1}^q \cdots p_{i_n}^q\right)\]
satisfies $a_{n+m }\geq a_n+a_m$ for all $n,m \geq 1$, and this is well known to imply the convergence of the sequence $(1/n)a_n$ to a limit in $(-\infty,+\infty]$. Observe that $\varphi^s(A) \geq \sigma_2(A)^s$ for all $A \in M_2(\mathbb{R})$. Since $A_1,A_2$ are invertible we have $\sigma_2(A_1),\sigma_2(A_2) \geq \varepsilon$ for some $\varepsilon>0$, and thus
\begin{align*}\sum_{i_1,\ldots,i_n=1}^N \varphi^s\left(A_{i_1}\cdots A_{i_n}\right)^{1-q}p_{i_1}^q \cdots p_{i_n}^q &\leq \sum_{i_1,\ldots,i_n=1}^N \sigma_2\left(A_{i_1}\cdots A_{i_n}\right)^{s(1-q)}p_{i_1}^q \cdots p_{i_n}^q\\
& \leq \sum_{i_1,\ldots,i_n=1}^N \varepsilon^{s(1-q)}p_{i_1}^q \cdots p_{i_n}^q\\
&=\varepsilon^{ns(1-q)}\left(p^q+(1-p)^q\right)^n \end{align*}
(where we have used the fact that $1-q$ is negative) so that the limit is finite. 

Let us show that $\mathbf{R}_q(A_1,A_2,p,s)$ is strictly increasing in $s$. We  note that $\varphi^{s+t}(A)\leq \varphi^s(A)\|A\|^t$ for all $s,t>0$ and for every matrix $A \in M_2(\mathbb{R})$. Taking $\theta:=\max\{\|A_1\|,\|A_2\|\} \in (0,1)$ it follows that for all $n \geq 1$
\begin{eqnarray*}\lefteqn{\sum_{i_1,\ldots,i_n=1}^N \varphi^{s+t}\left(A_{i_1}\cdots A_{i_n}\right)^{1-q}p_{i_1}^q \cdots p_{i_n}^q}& & \\
& \geq& \theta^{nt(1-q)} \sum_{i_1,\ldots,i_n=1}^N \varphi^{s}\left(A_{i_1}\cdots A_{i_n}\right)^{1-q}p_{i_1}^q \cdots p_{i_n}^q\end{eqnarray*}
and therefore
\[\mathbf{R}_q(A_1,A_2,p,s+t) \geq (1-q)t\log \theta + \mathbf{R}_q(A_1,A_2,p,s)>\mathbf{R}_q(A_1,A_2,p,s)\]
as required.
\end{proof}
Our interest in the previous lemma is due to the following consequence:
\begin{corollary}\label{co:cks}
For all invertible matrices $A_1,A_2$ such that $\|A_1\|,\|A_2\|<1$, all $p \in (0,1)$ and all $q>1$, we have
\begin{align}\label{eq:eek}\mathfrak{r}_q(A_1,A_2,p)&=\sup\left\{s>0 \colon \mathbf{R}_q(A_1,A_2,p,s)<0\right\}\\\nonumber
&=\inf\left\{s>0 \colon \mathbf{R}_q(A_1,A_2,p,s)>0\right\}.\end{align}
\end{corollary}
Before commencing the proof of Theorem \ref{th:fad}, let us briefly describe its strategy. The proofs of continuity of the affinity dimension $\mathfrak{s}$ given in \cite{FeSh14,Mo16} operate by defining the singular value pressure function
\[\mathbf{S}(A_1,\ldots,A_N,s):=\lim_{n \to \infty} \frac{1}{n}\log \sum_{i_1,\ldots,i_n=1}^N \varphi^s\left(A_{i_1}\cdots A_{i_n}\right)\]
and observing that for fixed invertible $A_1,\ldots,A_N$ with $\max \|A_i\|<1$ the function $s \mapsto \mathbf{S}(A_1,\ldots,A_N,s)$ is strictly decreasing, so that
\begin{align}\label{eq:quack}\mathfrak{s}(A_1,\ldots,A_N) &=\sup\{s>0 \colon \mathbf{S}(A_1,\ldots,A_N,s)>0\}\\\nonumber&=\inf\{s>0 \colon \mathbf{S}(A_1,\ldots,A_N,s)<0\}.\end{align}
The proofs then proceed by showing that for each fixed $s>0$ (or in the case of \cite{Mo16}, for a dense set of $s>0$) the function $(A_1,\ldots,A_N) \mapsto \mathbf{S}(A_1,\ldots,A_N,s)$ is continuous, and then deduce the continuity of $\mathfrak{s}$ via the formula \eqref{eq:quack}. The argument which we employ in proving Theorem \ref{th:fad} essentially converse to this: we demonstrate the existence of discontinuities in the function $(A_1,A_2) \mapsto \mathbf{R}_q(A_1,A_2,p,s)$ and show that they induce discontinuities in $\mathfrak{r}_q$ via the equation \eqref{eq:eek}. 

The origin of these discontinuities can be described informally as follows. Following \cite{Gu95}, let us define the \emph{lower spectral radius} of a pair of matrices $A_1,A_2$ to be the quantity
\begin{align*}\underline{\varrho}(A_1,A_2)&:=\lim_{n \to \infty} \min_{1 \leq i_1,\ldots,i_n \leq 2} \left\|A_{i_1}\cdots A_{i_n}\right\|^{\frac{1}{n}}\\
&=\inf_{n \geq 1} \min_{1 \leq i_1,\ldots,i_n \leq 2} \left\|A_{i_1}\cdots A_{i_n}\right\|^{\frac{1}{n}}.\end{align*}
The lower spectral radius is known to depend discontinuously on the matrix entries in general \cite[p.20]{Ju09}, and this phenomenon was investigated in depth by the author and J. Bochi in \cite{BoMo15}. This relates to $\mathbf{R}_q(A_1,A_2,p,s)$ as follows: if $0<s\leq 1$ then we may estimate
\begin{eqnarray*}\lefteqn{\mathbf{R}_q(A_1,A_2,p,s)}& & \\
& \leq& \limsup_{n \to \infty} \frac{1}{n}\log \left(\sum_{i_1,\ldots,i_n=1}^2 p_{i_1}^q\cdots p_{i_n}^q\max_{1 \leq j_1,\ldots,j_n \leq 2} \left\|A_{j_1}\cdots A_{j_n}\right\|^{s(1-q)}\right)\\
&=&\log(p^q+(1-p)^q) + s(1-q)\log\underline{\varrho}(A_1,A_2)\end{eqnarray*}
-- where the negativity of the exponent $s(1-q)$ has the critical effect of converting the maximum over all matrix products into a minimum --  and similarly on the other hand
\begin{eqnarray*}\lefteqn{\mathbf{R}_q(A_1,A_2,p,s)}  & &\\
&\geq& \liminf_{n \to \infty} \frac{1}{n}\log \left(\min_{1 \leq i_1,\ldots,i_n \leq 2} p_{i_1}^q\cdots p_{i_n}^q \cdot \max_{1 \leq j_1,\ldots,j_n \leq 2} \left\|A_{j_1}\cdots A_{j_n}\right\|^{s(1-q)}\right)\\
&=&q\log \min\{p,1-p\}+ s(1-q)\log\underline{\varrho}(A_1,A_2).\end{eqnarray*}
These estimates, despite their crudity, imply that if the discontinuity of the lower spectral radius $\underline{\varrho}$ at a particular pair of matrices is strong enough then it induces a discontinuity in the function $\mathbf{R}_q$, which can be exploited to deduce a discontinuity in the function $\mathfrak{r}_q$. Indeed, the examples of discontinuity of $\mathfrak{r}_q$ in Theorem \ref{th:fad}(i) and (ii) correspond directly with known examples of the discontinuity of the lower spectral radius, specifically Example 1.1 and Proposition 7.6 in \cite{BoMo15}. In the context of Theorem \ref{th:fad}(i) we can obtain sufficient control on the size of the discontinuity in $\underline{\varrho}$ without any assumptions on $p$ and $q$. In the context of Theorem \ref{th:fad}(ii) our much weaker control on the discontinuities of $\underline{\varrho}$ means that the above estimate is only useful if $q$ is large and $p$ to be close to $\frac{1}{2}$, which has the effect of bringing the quantities $\log (p^q+(1-p)^q)$ and $q\log (\min\{p,(1-p)\})$ closer together. In order to deal with more general $p$ the proof of Theorem \ref{th:fad}(ii) in fact applies a slightly finer estimate than that indicated above: for this we require a slightly strengthened statement of \cite[Proposition 7.6]{BoMo15}, which shows not only \emph{that} the lower spectral radius is discontinuous in certain places but also specifies \emph{how} it is discontinuous. We nonetheless emphasise that the conceptual origin of the discontinuity of $\mathfrak{r}_q$ in this paper is that it is a consequence of the discontinuity of the lower spectral radius.

\begin{proof}[Proof of Theorem \ref{th:fad}(i)]
Let $q>1$ and $0<\lambda<\frac{1}{2}$, and let $A_1,A_2,\delta$ be as in Theorem \ref{th:fad}(i). Throughout the proof we shall find it useful to write $p_1:=p$, $p_2:=(1-p)$ in order to simplify certain frequently-arising expressions.
We observe that by straightforward differentiation and minimisation with respect to $p$ one has $p^q+(1-p)^q \geq 2^{1-q}$ for every $q>1$ and $p \in (0,1)$. In particular, noting the hypothesis of Theorem \ref{th:fad}(i) and the negativity of $\log \lambda$,
\begin{equation}\label{eq:incidental}0 < \frac{q\log\min\{p,1-p\}}{(q-1)\log \sqrt{\lambda\delta}}<\frac{\log(p^q+(1-p)^q)}{(q-1)\log \lambda} \leq \frac{\log\frac{1}{2}}{\log \lambda}< 1.\end{equation}
Clearly $\|A_{i_1}\cdots A_{i_n}\|=\lambda^n$ for every $i_1,\ldots,i_n \in \{1,2\}$ and $n \geq 1$, so for every $s \in (0,1]$
\begin{align*}\mathbf{R}_q(A_1,A_2,p,s)&=\lim_{n \to \infty} \frac{1}{n}\log \left(\sum_{i_1,\ldots,i_n=1}^2 \lambda^{sn(1-q)} p_{i_1}^q \cdots p_{i_n}^q\right)\\
&=s(1-q)\log\lambda + \log\left(p^q+(1-p)^q\right).\end{align*}
In particular if $s<\log(p^q+(1-p)^q)/(q-1)\log \lambda \in (0,1]$ then $\mathbf{R}_q(B_1,B_2,p,s)<0$, so we have
\[\mathfrak{r}_q(A_1,A_2,p) \geq \frac{\log(p^q+(1-p)^q)}{(q-1)\log \lambda}\]
using Corollary \ref{co:cks}. Now fix an integer $k \geq 1$ and define $B_1:=A_1$ and 
\[B_{2}:=\lambda\left(\begin{array}{cc}\cos\frac{\pi}{2k} &-\sin\frac{\pi}{2k}\\\sin\frac{\pi}{2k} &\cos\frac{\pi}{2k}\end{array}\right)\]
so that
\[B_{2}^k=\left(\begin{array}{cc}0 &-\lambda^k\\\lambda^k &0\end{array}\right).\]
(We observe that by taking $k$ sufficiently large, $(B_1,B_2)$ may be taken as close to $(A_1,A_2)$ as desired.) Since we have
\[B_1^nB_{2}^kB_1^n = \left(\begin{array}{cc}\lambda^n &0\\0 &\delta^n\end{array}\right)\left(\begin{array}{cc}0 &-\lambda^k\\\lambda^k &0\end{array}\right)\left(\begin{array}{cc}\lambda^n &0\\0 &\delta^n\end{array}\right)= \left(\begin{array}{cc}\lambda^{n+k}\delta^n &0\\0 &\lambda^{n+k}\delta^n\end{array}\right)\]
for all $n \geq 1$ it follows that
\[\min_{1 \leq i_1,\ldots,i_{2n+k}\leq 2} \|B_{i_1}\cdots B_{i_n}\| \leq \lambda^{n+k}\delta^n\]
for every $n \geq 1$, and therefore
\[\lim_{n \to \infty} \left(\min_{1 \leq i_1,\ldots,i_{n}\leq 2} \|B_{i_1}\cdots B_{i_n}\| \right)^{\frac{1}{n}}\leq \sqrt{\lambda\delta}.\]
  Hence for $0<s \leq 1$
\begin{eqnarray*}\lefteqn{\mathbf{R}_q(B_1,B_2,p,s)} & & \\
& =& \lim_{n \to \infty}\frac{1}{n}\log \left(\sum_{i_1,\ldots,i_n=1}^2 \left\|B_{i_1}\cdots B_{i_n}\right\|^{s(1-q)}p_{i_1}^q \cdots p_{i_n}^q\right)\\
&\geq& \lim_{n \to \infty}\frac{1}{n}\log \left(\max_{1 \leq i_1,\ldots,i_n \leq 2}\left( \left\|B_{i_1}\cdots B_{i_n}\right\|^{s(1-q)}\right)\min\left\{p^{qn},(1-p)^{qn}\right\}\right)\\
&=& \lim_{n \to \infty}\frac{1}{n}\log \left(\left(\min_{1 \leq i_1,\ldots,i_n \leq 2} \left\|B_{i_1}\cdots B_{i_n}\right\|\right)^{s(1-q)}\min\left\{p^{qn},(1-p)^{qn}\right\}\right)\\
&\geq & \frac{s(1-q)}{2} \log(\lambda\delta) + \log\min\{p^q,(1-p)^q\}.\end{eqnarray*}
If $1 \geq s>q\log \min\{p,1-p\} / (q-1)\log \sqrt{\lambda\delta} \in (0,1)$ then this last term exceeds $0$ and therefore $\mathfrak{r}_q(B_1,B_2,p)\geq s$. Hence in view of Corollary \ref{co:cks}
\[\mathfrak{r}_q(B_1,B_2,p) \leq \frac{q \log \min\{p,1-p\}}{(q-1)\log \sqrt{\lambda\delta}}.\]
As was previously noted, by taking the integer $k$ in the definition of $B_2$ arbitrarily large we may take $(B_1,B_2)$ as above arbitrarily close to $(A_1,A_2)$, and it follows that
\begin{align*}\liminf_{(B_1,B_2) \to (A_1,A_2)} \mathfrak{r}_q(B_1,B_2,p) &\leq  \frac{2q \log \min\{p,1-p\}}{(q-1)\log (\lambda\delta)}\\&<\frac{\log(p^q+(1-p)^q)}{(q-1)\log \lambda}\leq \mathfrak{r}_q(A_1,A_2,p) \end{align*}
where we have used \eqref{eq:incidental}, so that $\mathfrak{r}_q$ is discontinuous at $(A_1,A_2)$ as claimed. This completes the proof of (i).
\end{proof}

The proof of (ii) uses closely analogous estimates, but we require an additional result relating to Conjecture \ref{co:bf}. The following result is a more specialised reworking of one half of \cite[Proposition 7.6]{BoMo15}.
\begin{lemma}\label{le:bf}
Suppose that Conjecture \ref{co:bf} is true. Then there exist $\varepsilon,\kappa>0$ and a set $X \subset M_2(\mathbb{R})^2$ with positive Lebesgue measure such that for all $(A_1,A_2) \in X$ we have
\[\|A_1\|,\|A_2\|<\frac{1}{2},\]
\[ \underline{\varrho}(A_1,A_2)\geq e^{-\kappa},\]
but such that in every open neighbourhood of $(A_1,A_2)$ we may find $(B_1,B_2)$  such that for a certain integer $k \geq 1$ depending on $(B_1,B_2)$
\[\lim_{n \to \infty}\left\|B_1^nB_2^kB_1^n\right\|^{\frac{1}{2n+k}}=\underline{\varrho}(B_1,B_2) \leq e^{-\varepsilon-\kappa}. \]
\end{lemma}
\begin{proof}
Let us first define
\[Z:=\left\{(\alpha H, \beta R) \colon H \in \mathcal{H}, R \in \mathcal{E}\text{ and }0<\alpha<\beta<\frac{1}{2\|H\|}.\right\}\]
Clearly this is an open subset of $M_2(\mathbb{R})^2$, and for every $(A_1,A_2) \in Z$ both of the matrices $A_i$ have positive determinant and have norm strictly less than one half. By the hypothesis that Conjecture \ref{co:bf} is true, the set of all $(\alpha H, \beta R) \in Z$ such that $(H,R)$ is resistant has full Lebesgue measure in $Z$, and hence in particular has positive Lebesgue measure in $M_2(\mathbb{R})^2$.

We first claim that for all $(A_1,A_2)=(\alpha H, \beta R)$ such that $(H,R)$ is resistant we have $\underline{\varrho}(A_1,A_2)>\sqrt{\det A_1}$. Indeed, suppose that $(H,R)$ is $(c,\lambda,\varepsilon)$-resistant where $c,\varepsilon>0$ and $\lambda>1$. If $A_{i_1},\ldots,A_{i_n}$ contains at most $\varepsilon n$ instances of $A_2$ then we have
\[\left\|A_{i_1}\cdots A_{i_n}\right\| \geq c\lambda^n \alpha^n\]
since $\|A_{i_1}\cdots A_{i_n}\|$ is at least $\alpha^n$ times the norm of a product of $n$ of the matrices $H$,$R$ in which at most $\varepsilon n$ matrices are equal to $R$. On the other hand if the product $A_{i_1}\cdots A_{i_n}$ contains at least $\varepsilon n$ instances of $A_2$ then we have
\[\left\|A_{i_1}\cdots A_{i_n}\right\| \geq \sqrt{|\det A_{i_1}\cdots A_{i_n}|} \geq \alpha^{(1-\varepsilon)n}\beta^{\varepsilon n},\]
and therefore
\[\underline{\varrho}(A_1,A_2) =\lim_{n \to \infty} \inf_{1 \leq i_1,\ldots,i_n \leq 2} \left\|A_{i_1}\cdots A_{i_n}\right\|^{\frac{1}{n}} \geq \min\left\{\lambda \alpha, \left(\frac{\beta}{\alpha}\right)^\varepsilon \alpha\right\}>\alpha=\sqrt{\det A_1}\]
as claimed.

We next claim that for every $(A_1,A_2)=(\alpha H, \beta R)\in Z$, we may in every open neighbourhood of $A_2$ find a matrix $B_2$ such that for some integer $k \geq 1$ depending on $B_2$,
\[ \lim_{n \to \infty} \left\|A_1^nB_2^kA_1^n\right\|^\frac{1}{2n+k}=\sqrt{\det A_1}.\]
(We note that in this case necessarily $\underline{\varrho}(A_1,B_2)=\sqrt{\det A_1}$, since clearly any product of $n$ of those two matrices is bounded below in norm by the square root of the determinant, which in turn is bounded below by $(\det A_1)^n$.)
To show this it is sufficient to show that for any fixed $H \in \mathcal{H}$ and $R \in \mathcal{E}$, we may in every open neighbourhood of $R$ find a matrix $R'$ such that for some integer $k \geq 1$ depending on $R'$,
\begin{equation}\label{eq:bo} \lim_{n \to \infty} \left\|H^n(R')^kH^n\right\|^{\frac{1}{2n+k}}=1.\end{equation}
Let us prove this statement. Given $(H,R)\in\mathcal{H}\times\mathcal{E}$ let $\lambda>1$ denote the larger eigenvalue of $H$, and let $u$ and $v$ be eigenvectors of $H$ corresponding respectively to the eigenvalues $\lambda$ and $\lambda^{-1}$. Since $R$ has non-real eigenvalues it is conjugate to a rotation through some particular angle $\theta$. It is easy to see that this implies that in any neighbourhood of $R$ we may find a matrix $R'$, conjugate to a rotation through a different angle, such that $(R')^ku=\gamma v$ for some nonzero real number $\gamma$ and integer $k \geq 1$. In the basis $(u,v)$ we have
\[H^n(R')^kH^n = \left(\begin{array}{cc}\lambda^n&0\\0&\lambda^{-n}\end{array}\right)\left(\begin{array}{cc}0&\delta \\\gamma &\epsilon\end{array}\right)\left(\begin{array}{cc}\lambda^n&0\\0&\lambda^{-n}\end{array}\right)=\left(\begin{array}{cc}0&\delta \\\gamma &\lambda^{-2n}\epsilon\end{array}\right)\]
for some real numbers $\delta,\epsilon$, where the first column of $(R')^k$ reflects the fact that $(R')^ku=\gamma v$. In particular $\|H^n(R')^kH^n\|$ is bounded independently of $n$, and this yields \eqref{eq:bo}.

Summarising the proof so far, we have shown that there is a full-measure subset $Z_1$ of the open set $Z\subset M_2(\mathbb{R})^2$ such that every $(A_1,A_2)\in Z_1$ satisfies $\|A_1\|,\|A_2\|<\frac{1}{2}$ and $\underline{\varrho}(A_1,A_2)>\sqrt{\det A_1}$, and has the property that in every open neighbourhood of $A_2$ we may find $B_2$ such that for some integer $k \geq 1$,
\[\lim_{n \to \infty}\left\|A_1^nB_2^kA_1^n\right\|^{\frac{1}{2n+k}}=\underline{\varrho}(A_1,B_2)=\sqrt{\det A_1}.\]
So, let us choose $\kappa>0$ such that the set
\[Z_2:=\left\{(A_1,A_2) \in Z_1 \colon \underline{\varrho}(A_1,A_1) \geq e^{-\kappa}>\sqrt{\det A_1} \right\}\]
has positive Lebesgue measure, and choose $\varepsilon>0$ such that the set
\[X:=\left\{(A_1,A_2) \in Z_2 \colon \sqrt{\det A_1} <e^{-\kappa-\varepsilon}\right\}\]
has positive Lebesgue measure. The proof is complete. 
\end{proof}
\emph{Remark.} In order to improve Theorem \ref{th:fad}(ii) so as to allow arbitrary $q>1$ it would be sufficient to be able to choose the set $X$ in such a way that the ratio $\varepsilon/\kappa$ is made arbitrarily large. In effect, this asks that we should be able to reduce the second singular value of $A_1$ arbitrarily far without simultaneously reducing $\underline{\varrho}(A_1,A_2)$ by a comparable amount:  in Theorem \ref{th:fad}(i), this effect is achieved by the simple expedient of reducing $\delta$.

\begin{proof}[Proof of Theorem \ref{th:fad}(ii)]
Let $X$, $\varepsilon$, $\kappa$ be as in Lemma \ref{le:bf} and choose $Q:=1+\frac{\kappa}{\varepsilon}>1$. For all $p \in [\frac{1}{2},1)$ and $q>Q$ we have
\[\left(p^q+(1-p)^q\right)^{\frac{1}{q}}<p^{\frac{\kappa}{\kappa+\varepsilon}}\]
since for each fixed $q>1$ the former expression is a convex function of $p$, the latter is a concave function of $p$, the two functions agree at $p=1$ and the former function is strictly less than the latter at $p=\frac{1}{2}$. Rearranging we find that for all such $p$ and $q$
\[\frac{\log p^q}{\log (p^q+(1-p)^q)} <1+\frac{\varepsilon}{\kappa}.\]
Since clearly $e^{-\kappa} \leq \underline{\varrho}(A_1,A_2) \leq \max\{\|A_1\|,\|A_2\|\} <\frac{1}{2}$ we have $\kappa>\log 2$ and therefore 
\begin{equation}\label{eq:noncidental}0<\frac{\log(p^q)}{(1-q)(\varepsilon + \kappa)}<\frac{\log (p^q+(1-p)^q)}{(1-q)\kappa} <\frac{\log (p^q+(1-p)^q)}{(1-q)\log 2} \leq 1\end{equation}
for all $p \in [\frac{1}{2},1)$ and $q>Q$, where we have reused the elementary inequality $p^q+(1-p)^q \geq 2^{1-q}$ which was similarly applied in (i). We will show that for all such $p$ and $q$, every point of $X$ is a point of discontinuity of the map $(A_1,A_2)\mapsto \mathfrak{r}_q(A_1,A_2,p)$.

Let us therefore fix $p$ and $q$ and take $(A_1,A_2) \in X$. For every $s \in (0,1]$ we have
\begin{eqnarray*}\lefteqn{\mathbf{R}_q(A_1,A_2,p,s)} & &\\
&=&\lim_{n \to \infty}\frac{1}{n}\log \left(\sum_{i_1,\ldots,i_n=1}^2 \|A_{i_1}\cdots A_{i_n}\|^{s(1-q)} p_{i_1}^q \cdots p_{i_n}^q\right)\\
&\leq& \lim_{n \to \infty}\frac{1}{n}\log\left( \max_{1 \leq i_1,\ldots,i_n \leq 2} \left(\|A_{i_1}\cdots A_{i_n}\|^{s(1-q)}\right) \sum_{j_1,\ldots,j_n=1}^2 p_{j_1}^q \cdots p_{j_n}^q\right)\\
&=& \lim_{n \to \infty}\frac{1}{n}\log\left( \left( \min_{1 \leq i_1,\ldots,i_n \leq 2} \|A_{i_1}\cdots A_{i_n}\|\right)^{s(1-q)} \sum_{j_1,\ldots,j_n=1}^2 p_{j_1}^q \cdots p_{j_n}^q\right)\\
&=&s(1-q)\log\underline{\varrho}(A_1,A_2) +\log\left(p^q+(1-p)^q\right)\\
&\leq&s(q-1)\kappa +\log\left(p^q+(1-p)^q\right).\end{eqnarray*}

It follows that if $s< \log(p^q+(1-p)^q) / (1-q)\kappa < 1$ then $\mathbf{R}_q(A_1,A_2,p,s)$ is negative, and hence by Corollary \ref{co:cks}
\[\mathfrak{r}_q(A_1,A_2,p) \geq \frac{\log (p^q+(1-p)^q)}{(1-q)\kappa}.\]
On the other hand we may take $(B_1,B_2)$ arbitrarily close to $(A_1,A_2)$ such that
\[\lim_{n \to \infty}\left\|B_1^nB_2^kB_1^n\right\|^{\frac{1}{2n+k}}=\underline{\varrho}(B_1,B_2) < e^{-\varepsilon-\kappa} \]
for some integer $k \geq 1$. In particular we have
\begin{eqnarray*}\lefteqn{\lim_{n \to \infty} \frac{1}{n}\log \max_{1 \leq i_1,\ldots,i_n \leq 2} \left(\left\|B_{i_1}\cdots B_{i_n}\right\|^{s(1-q)}p_{i_1}^q\cdots p_{i_n}^q\right)}& &\\
&\leq& \lim_{n \to \infty} \frac{1}{n}\log \max_{1 \leq i_1,\ldots,i_n \leq 2} \left(\left\|B_{i_1}\cdots B_{i_n}\right\|\right)^{s(1-q)}\\
& & +\lim_{n \to \infty}\frac{1}{n} \log \max_{1 \leq i_1,\ldots,i_n\leq 2}\left(p_{i_1}^q\cdots p_{i_n}^q\right)\\
&=& \lim_{n \to \infty} \frac{s(1-q)}{n}\log \min_{1 \leq i_1,\ldots,i_n \leq 2} \left\|B_{i_1}\cdots B_{i_n}\right\| + q\log p\\
&=&s(1-q)\log\underline{\varrho}(B_1,B_2)+q\log p\end{eqnarray*}
since $p_1:=p \geq \frac{1}{2}\geq 1-p=p_2$, but also
\begin{eqnarray*}\lim_{n \to \infty} \lefteqn{\frac{1}{n}\log \max_{1 \leq i_1,\ldots,i_n \leq 2} \left(\left\|B_{i_1}\cdots B_{i_n}\right\|^{s(1-q)}p_{i_1}^q\cdots p_{i_n}^q\right)}& & \\
&\geq& \lim_{n \to \infty} \frac{1}{2n+k} \log \left(\left\|B_1^nB_2^kB_1^n\right\|^{s(1-q)}p_1^np_2^kp_1^n\right)\\
&=&s(q-1)\log\underline{\varrho}(B_1,B_2) + q\log p,\end{eqnarray*}
and we conclude that the limit
\[\lim_{n \to \infty} \frac{1}{n}\log \max_{1 \leq i_1,\ldots,i_n \leq 2} \left(\left\|B_{i_1}\cdots B_{i_n}\right\|^{s(1-q)}p_{i_1}^q\cdots p_{i_n}^q\right)\]
is equal to $s(q-1)\log \underline{\varrho}(B_1,B_2)+q\log p$.
It follows that for all $s \in (0,1]$
\begin{align*}\mathbf{R}_q(B_1,B_2,p,s) &= \lim_{n \to \infty}\frac{1}{n}\log \left(\sum_{i_1,\ldots,i_n=1}^2 \left\|B_{i_1}\cdots B_{i_n}\right\|^{s(1-q)}p_{i_1}^q \cdots p_{i_n}^q\right)\\
&\geq \lim_{n \to \infty}\frac{1}{n}\log \left(\max_{1 \leq i_1,\ldots,i_n \leq 2}\left( \left\|B_{i_1}\cdots B_{i_n}\right\|^{s(1-q)}p_{i_1}^{q}\cdots p_{i_n}^q\right)\right)\\
&=s(1-q)\log \underline{\varrho}(B_1,B_2) +q\log p\\
&> s(q-1)(\varepsilon+\kappa)+q\log p.\end{align*}
If $1 \geq s>(q\log p)/(1-q)(\varepsilon+\kappa) \in (0,1)$ then $\mathbf{R}_q(B_1,B_2,p,s)>0$, and therefore
\[\mathfrak{r}_q(B_1,B_2,p) \leq \frac{\log (p^q)}{(1-q)(\varepsilon+\kappa)}\]
by Corollary \ref{co:cks}.
Hence
\begin{align*}\liminf_{(B_1,B_2)\to(A_1,A_2)} \mathfrak{r}_q(B_1,B_2,p)&\leq \frac{q\log p}{(1-q)(\varepsilon+\kappa)}\\
&<\frac{\log (p^q+(1-p)^q)}{(1-q)\kappa}\\&\leq \mathfrak{r}_q(A_1,A_2,p)\end{align*}
using \eqref{eq:noncidental}, and $(B_1,B_2) \mapsto \mathfrak{r}_q(B_1,B_2,p)$ is discontinuous at $(A_1,A_2)$ as claimed.
\end{proof}
\section{Acknowledgements}
The author was supported by EPSRC grant EP/L026953/1.
\bibliographystyle{siam}
\bibliography{renyi}

\begin{thebibliography}{10}

\bibitem{AlSeUn11}
{\sc J.~Allen, B.~Seeger, and D.~Unger}, {\em On the size of the resonant set
  for the products of {$2\times 2$} matrices}, Involve, 4 (2011), pp.~157--166.

\bibitem{AvRo09}
{\sc A.~Avila and T.~Roblin}, {\em Uniform exponential growth for some {${\rm
  SL}(2,\Bbb R)$} matrix products}, J. Mod. Dyn., 3 (2009), pp.~549--554.

\bibitem{Ba07}
{\sc K.~Bara{\'n}ski}, {\em Hausdorff dimension of the limit sets of some
  planar geometric constructions}, Adv. Math., 210 (2007), pp.~215--245.

\bibitem{Be84}
{\sc T.~Bedford}, {\em Crinkly curves, {M}arkov partitions and box dimensions
  in self-similar sets}, 1984.
\newblock PhD thesis, University of Warwick.

\bibitem{BoFa06}
{\sc J.~Bochi and B.~Fayad}, {\em Dichotomies between uniform hyperbolicity and
  zero {L}yapunov exponents for {${\rm SL}(2,{\Bbb R})$} cocycles}, Bull. Braz.
  Math. Soc. (N.S.), 37 (2006), pp.~307--349.

\bibitem{BoMo15}
{\sc J.~Bochi and I.~D. Morris}, {\em Continuity properties of the lower
  spectral radius}, Proc. Lond. Math. Soc. (3), 110 (2015), pp.~477--509.

\bibitem{Fa88}
{\sc K.~J. Falconer}, {\em The {H}ausdorff dimension of self-affine fractals},
  Math. Proc. Cambridge Philos. Soc., 103 (1988), pp.~339--350.

\bibitem{Fa97}
\leavevmode\vrule height 2pt depth -1.6pt width 23pt, {\em Techniques in
  fractal geometry}, John Wiley \& Sons, Ltd., Chichester, 1997.

\bibitem{Fa99}
\leavevmode\vrule height 2pt depth -1.6pt width 23pt, {\em Generalized
  dimensions of measures on self-affine sets}, Nonlinearity, 12 (1999),
  pp.~877--891.

\bibitem{Fa03}
\leavevmode\vrule height 2pt depth -1.6pt width 23pt, {\em Fractal geometry:
  mathematical foundations and applications}, John Wiley \& Sons, Inc.,
  Hoboken, NJ, second~ed., 2003.

\bibitem{Fa10}
\leavevmode\vrule height 2pt depth -1.6pt width 23pt, {\em Generalized
  dimensions of measures on almost self-affine sets}, Nonlinearity, 23 (2010),
  pp.~1047--1069.

\bibitem{Fa13}
\leavevmode\vrule height 2pt depth -1.6pt width 23pt, {\em Dimensions of
  self-affine sets: a survey}, in Further developments in fractals and related
  fields, Trends Math., Birkh\"auser/Springer, New York, 2013, pp.~115--134.

\bibitem{FaKr08}
{\sc B.~Fayad and R.~Krikorian}, {\em Exponential growth of product of matrices
  in {${\rm SL}(2,\Bbb R)$}}, Nonlinearity, 21 (2008), pp.~319--323.

\bibitem{FeSh14}
{\sc D.-J. Feng and P.~Shmerkin}, {\em Non-conformal repellers and the
  continuity of pressure for matrix cocycles}, Geom. Funct. Anal., 24 (2014),
  pp.~1101--1128.

\bibitem{Fr12}
{\sc J.~M. Fraser}, {\em On the packing dimension of box-like self-affine sets
  in the plane}, Nonlinearity, 25 (2012), pp.~2075--2092.

\bibitem{Gu95}
{\sc L.~Gurvits}, {\em Stability of discrete linear inclusion}, Linear Algebra
  Appl., 231 (1995), pp.~47--85.

\bibitem{Hu81}
{\sc J.~E. Hutchinson}, {\em Fractals and self-similarity}, Indiana Univ. Math.
  J., 30 (1981), pp.~713--747.

\bibitem{JoPoSi07}
{\sc T.~Jordan, M.~Pollicott, and K.~Simon}, {\em Hausdorff dimension for
  randomly perturbed self affine attractors}, Comm. Math. Phys., 270 (2007),
  pp.~519--544.

\bibitem{Ju09}
{\sc R.~Jungers}, {\em The joint spectral radius: theory and applications},
  vol.~385 of Lecture Notes in Control and Information Sciences,
  Springer-Verlag, Berlin, 2009.

\bibitem{LaGa92}
{\sc S.~P. Lalley and D.~Gatzouras}, {\em Hausdorff and box dimensions of
  certain self-affine fractals}, Indiana Univ. Math. J., 41 (1992),
  pp.~533--568.

\bibitem{La95}
{\sc K.-S. Lau}, {\em Self-similarity, {$L^p$}-spectrum and multifractal
  formalism}, in Fractal geometry and stochastics ({F}insterbergen, 1994),
  vol.~37 of Progr. Probab., Birkh\"auser, Basel, 1995, pp.~55--90.

\bibitem{Mc84}
{\sc C.~McMullen}, {\em The {H}ausdorff dimension of general {S}ierpi\'nski
  carpets}, Nagoya Math. J., 96 (1984), pp.~1--9.

\bibitem{Mo16}
{\sc I.~D. Morris}, {\em An inequality for the matrix pressure function and
  applications}.
\newblock arXiv preprint 1507.00642, 2015.

\bibitem{So98}
{\sc B.~Solomyak}, {\em Measure and dimension for some fractal families}, Math.
  Proc. Cambridge Philos. Soc., 124 (1998), pp.~531--546.

\end{thebibliography}
\end{document}